\documentclass[12pt,a4paper]{amsart}
\usepackage{amssymb,amscd,  latexsym, graphicx, mathrsfs, enumerate}
\usepackage{color}
\newcommand{\nc}{\newcommand}
\numberwithin{equation}{section}
\newtheorem{theorem}{Theorem}[section]
\newtheorem{prop}[theorem]{Proposition}
\newtheorem{importnota}[theorem]{Important Notation}
\newtheorem{prblm}[theorem]{Problem}
\newtheorem{notation}[theorem]{Notation}
\newtheorem{caution}[theorem]{Caution}
\newtheorem{remark}[theorem]{Remark}
\newtheorem{lemma}[theorem]{Lemma}
\newtheorem{construction}[theorem]{Construction}
\newtheorem{corollary}[theorem]{Corollary}
\newtheorem{example}[theorem]{Example}
\newtheorem{conclusion}[theorem]{Conclusion}
\newtheorem{triviality}[theorem]{Triviality}
\newtheorem{proto}[theorem]{Prototype Quasifibration}
\newtheorem{cauex}[theorem]{Cautionary Example}
\newtheorem{propositiondef}[theorem]{Proposition-Definition}
\newtheorem{subth}{Nuisance}[theorem]
\newtheorem{ssubth}{ }[subth]
\newtheorem{conjecture}[theorem]{Conjecture}
\newtheorem{sidest}[theorem]{Side Story}
\newtheorem{miniexample}[theorem]{Example}
\theoremstyle{definition}
\newtheorem{defin}[theorem]{Definition}
\nc\tri[1]{\begin{triviality}}
\nc\side[1]{\begin{sidest}}
\nc\conj[1]{\begin{conjecture}}
\nc\prodef[1]{\begin{propositiondef}}
\nc\prt[1]{\begin{proto}}
\nc\lem[1]{\begin{lemma}}
\nc\sblm[1]{\begin{sublemma}}
\nc\pro[1]{\begin{prop}}
\nc\thm[1]{\begin{theorem}}
\nc\cor[1]{\begin{corollary}}
\nc\dfn[1]{\begin{defin}}
\nc\sthm[1]{\begin{subth}}
\nc\exm[1]{\begin{example}}
\nc\miniexm[1]{\begin{miniexample}}
\nc\plm[1]{\begin{prblm}}
\nc\rmk[1]{\begin{remark}}
\nc\subrmk[1]{\begin{subremark}}
\nc\ntn[1]{\begin{notation}}
\nc\cau[1]{\begin{caution}}
\nc\imn[1]{\begin{importnota}}
\nc\cax[1]{\begin{cauex}}
\nc\con[1]{\begin{construction}}
\nc\ssthm[1]{\begin{ssubth}}
\nc\cnc[1]{\begin{conclusion}}
\nc\elem{\end{lemma}}
\nc\esblm{\end{sublemma}}
\nc\eside{\end{sidest}}
\nc\econj{\end{conjecture}}
\nc\eprodef{\end{propositiondef}}
\nc\eprt{\end{proto}}
\nc\ethm{\end{theorem}}
\nc\ecor{\end{corollary}}
\nc\edfn{\end{defin}}
\nc\esthm{\end{subth}}
\nc\epro{\end{prop}}
\nc\etri{\end{triviality}}
\nc\eexm{\end{example}}
\nc\eminiexm{\end{miniexample}}
\nc\ermk{\end{remark}}
\nc\subermk{\end{subremark}}
\nc\eplm{\end{prblm}}
\nc\ecau{\end{caution}}
\nc\ecax{\end{cauex}}
\nc\eimn{\end{importnota}}
\nc\entn{\end{notation}}
\nc\econ{\end{construction}}
\nc\ecnc{\end{conclusion}}
\nc\essthm{\end{ssubth}}

\numberwithin{equation}{section}

\newcommand{\ZZ}{{\bf Z}}
\newcommand{\QQ}{{\bf Q}}

\newcommand{\CC}{{\bf C}}

\newcommand{\PP}{{\bf P}}

\newcommand{\lra}{\longrightarrow}

\newcommand{\cD}{{\mathcal D}}

\newcommand{\cH}{{\mathcal H}}

\newcommand{\C}{\mathbb{C}}

\newcommand{\Q}{\mathbb{Q}}

\newcommand{\Z}{\mathbb{Z}}

\newcommand{\ds}{\displaystyle}

\newcommand{\tcD}{\tilde{\mathcal D}}

\newcommand{\bes}{\begin{equation*}}
\newcommand{\ees}{\end{equation*}}

\title[On intermediate Jacobians of cubic threefolds]
{On intermediate Jacobians of cubic threefolds admitting an automorphism of order five}
\author{Bert van Geemen and Takuya Yamauchi}
\dedicatory{To  Eduard Looijenga on his 69th birthday}
\thanks{The second author
is partially supported by JSPS Grant-in-Aid for Scientific Research (C) No.15K04787.}
\subjclass[2010]{}

\address{Bert van Geemen \\
Dipartimento di Matematica, Universit`a di Milano\\
Via Saldini 50, I-20133 Milano, Italy }
\email{lambertus.vangeemen@unimi.it}

\address{Takuya Yamauchi \\
Department of mathematics, Faculty of Education\\
Kagoshima University\\
Korimoto 1-20-6 Kagoshima 890-0065, JAPAN}
\email{yamauchi@edu.kagoshima-u.ac.jp}

\begin{document}
\begin{abstract}Let $k$ be a field of characteristic zero containing a primitive fifth root of unity.   
Let $X/k$ be a smooth cubic threefold with an automorphism of order five,
then we observe that over a finite extension of the field
actually the dihedral group $D_5$ is a subgroup of ${\rm Aut}(X)$. 
We find that the intermediate Jacobian $J(X)$ of $X$  
is isogenous to the product of an elliptic curve $E$ and the self-product of an abelian surface $B$ with real multiplication by $\QQ(\sqrt{5})$.
We give explicit models of some algebraic curves related to the  construction of 
$J(X)$ as a Prym variety. 
This includes a two parameter family of curves of genus 2 whose 
Jacobians are isogenous to the abelian surfaces mentioned as above. 
\end{abstract}
\maketitle

\section*{Introduction}\label{intro}
Let $k$ be a field of characteristic zero that contains  
a primitive fifth root  of unity $\zeta$.   
Let $X$ be a smooth cubic threefold  in $\mathbb{P}^4$ 
with an automorphism of order five over $k$.  
There is a two-dimensional family of such threefolds.
From the explicit equations provided by the paper \cite{victor} we deduce that,
over a finite extension of $k$, the dihedral group $D_5$ acts on $X$. 
It is then not hard to establish that 
the intermediate Jacobian $J(X)$ of $X$, a principally polarized abelian fivefold,
is isogenous to the product of an elliptic curve $E$ and the self-product of an
abelian surface $B$ with $\QQ(\sqrt{5})\subset {\rm End}_{\overline{k}}(B)\otimes_\Z\Q$. 

To study the abelian varieties $E$ and $B$ in more detail,
we use that $J(X)$ is isomorphic to the Prym variety 
of an etale double cover of curves. 
In fact, for a general line $l$ on $X$, the curve $\cH_l$ which parametrizes the lines 
on $X$ meeting $l$ has a fixed point free involution $\iota_l$. 
We found a line $l$ such that the automorphisms of $X$ in $D_5$ 
induce automorphisms on the curve $\cH_l$, commuting with the covering involution $\iota_l$.
In this way we get an action of the group $D_5$ on the Prym variety 
which agrees with the action of $D_5$ on $J(X)$. 
Given $X$, we can then identify, up to isogeny, the elliptic curve $E$
with an explicit quotient of $\cH_l$. The abelian surface $B$ is isogenous to 
the Prym variety of certain explicit double cover of a genus two curve by a genus four curve 
related to $\cH_l$.
Finally we identify the isogeny class of $B$ with the isogeny class of the Jacobian of a certain 
explicit genus two curve. This Jacobian has $\Z[\frac{1+\sqrt{5}}{2}]$ in its
endomorphism ring and the explicit equation for the 
curve might be of independent interest.

Using the results from J.D.\ Achter \cite{Achter}, one can obtain similar 
results over $\ZZ[\zeta,1/10]$. 


\section{A standard model for $X$ and decomposing $J(X)$}

\subsection{Outline}
We start by finding a nice model for the cubic threefolds
with an automorphism of order five given in \cite{victor}. 
We then observe that they also admit an involution and that the dihedral group $D_5$ acts on 
such cubic threefolds. 
We use the $D_5$-action to decompose the intermediate Jacobian in Proposition
\ref{finedecij}.

\subsection{A standard model}\label{stmo}
Let $k$ be a field, with $char(k)=0$,  which contains a primitive fifth root of unity $\zeta$. 
From \cite[Thm 3.5]{victor} we know that a smooth cubic threefold $X\subset \mathbb{P}^4$ over $k$ 
that admits an automorphism $\alpha_X$
over $k$ of order five has an equation $F=F_{\underline{a}}$: 
$$
F_{\underline{a}}=a_1x_0^3 + a_2x_0x_1x_4 + a_3x_0x_2x_3 + a_4x_1^2x_3 + a_5x_1x_2^2 + a_6x_2x_4^2+a_7x_3^2x_4\,=\,0
$$
where $\underline{a}=(a_1,\ldots,a_7)\in k^7$  
and the automorphism acts as 
$$
\alpha_X:\;(x_0:x_1:x_2:x_3:x_4)\,\longmapsto\, 
(x_0:\zeta x_1:\zeta^2x_2:\zeta^3x_3:\zeta^4x_4)~.
$$
We first study the condition on $F_{\underline{a}}$ 
to give a smooth cubic threefold in terms of the seven 
parameters.  
For $\underline{a}\in k^7$, we define a homogeneous polynomial by 
\begin{equation}
\begin{array}{rl}
\Delta(\underline{a}):=& a^4_2 a^5_3 a_4 a_6 + 8 a_1 a^2_2 a^4_3 a^2_4 a^2_6 + 16 a^2_1 a^3_3 a^3_4 a^3_6 + 
 a^5_2 a^4_3 a_5 a_7 + 15 a_1 a^3_2 a^3_3 a_4 a_5 a_6 a_7+\\
&  
 12 a^2_1 a_2 a^2_3 a^2_4 a_5 a^2_6 a_7 + 8 a_1 a^4_2 a^2_3 a^2_5 a^2_7 +
12 a^2_1 a^2_2 a_3 a_4 a^2_5 a_6 a^2_7 + 27 a^3_1 a^2_4 a^2_5 a^2_6 a^2_7+ \\ 
&16 a^2_1 a^3_2 a^3_5 a^3_7.
\end{array}
\end{equation}
Put 
\begin{equation}
D(\underline{a})=a_1 a_4 a_5 a_6 a_7\Delta(\underline{a}). 
\end{equation}
Put $$
P_1=(1:0:0:0:0),\quad P_4=(0:1:0:0:0),\quad P_5=(0:0:1:0:0)~,
$$
$$
P_6=(0:0:0:0:1),\quad P_7=(0:0:0:1:0)~.
$$
Then we have 

\begin{lemma}\label{discX}
The polynomial $F_{\underline{a}}$ defines a 
smooth cubic threefold if and only if $D(\underline{a})\not= 0$. 
\end{lemma}
\begin{proof}
Assume $a_i=0$ for some $i\in \{1,4,5,6,7\}$, then it is easy to see that $P_i\in X$ gives a 
singular  point. So we may assume that $a_i\neq 0$ for any $i\in \{1,4,5,6,7\}$.  
Let $I=J_{F_{\underline{a}}}$ be the homogeneous ideal in $k[x_0,\ldots,x_4]$ 
generated by the partial derivatives 
${\partial F_{\underline{a}}}/{\partial x_j}$ for  $j=0,\ldots,4$. 
By using a Groebner Basis, one can show that  
$$
a_{n_i}\Delta(\underline{a})x^6_i\in I,\qquad i=0,\ldots,4,
$$
where $(n_0,n_1,n_2,n_3,n_4)=(1,4,5,7,6)$.
So if $D(\underline{a})\not= 0$, then the radical of $I$
in $\overline{k}[x_0,\ldots,x_4]$ contains 
$(x_0,\ldots,x_4)$, hence $V(I)=\emptyset$.
Therefore $F_{\underline{a}}$ defines to a smooth cubic threefold. 

Conversely, assume that $X$ is smooth. Let $U_i$ be the open subset of $\mathbb{P}^4$ defined by 
$x_i\not=0$. Then one can check that 
the ideal corresponding to $V(I)\cap U_i$ contains $a_{n_i}\Delta(\underline{a})$ as the unique constant 
(up to scalar). 
Hence $a_{n_i}\Delta(\underline{a})$ should be non-zero for any $i$. This proves the lemma. 
\end{proof}

\subsection{A standard form} From Lemma \ref{discX} we see that
the coefficients $a_1,a_4,a_5,a_6,a_7$ must all be non-zero for a smooth cubic. Over an algebraically closed field one can make changes of the  coordinates which reveal that
a smooth cubic threefold with an automorphism of order five actually has 
$D_5$ in its automorphism group.

\begin{lemma}\label{lemmaxab}\label{d5x}
Let $X_{\underline{a}}$ be a smooth cubic threefold 
defined by $F_{\underline{a}}$ where $\underline{a}\in k^7$.
Assume that $a_2\not=0$ or $a_3\not=0$. 
Then there is a change of coordinates in $\PP^4_{\bar{k}}$
such that $X_{\underline{a}}$ is isomorphic to 
$$
X_{a,b}:\qquad F_{a,b}\,:=\, xu^2+2yuv+zv^2\;+\;2z^2u+ 2x^2v\;+\;
ay^3 + bxyz~=\,0
$$
for some $a,b\in \bar{k}$. Moreover, $X_{a,b}$ is smooth if and only if
\begin{equation}\label{disc2}
D(a,b):=\frac{1}{64}D(a, b, 2, 2, 1, 2, 1)=a\Delta(a,b)\not=0,
\end{equation}
$$
\Delta(a,b):=512 a^2 + 27 a^3 + 48 a^2 b + 128 a b^2 +  6 a^2 b^2 + 30 a b^3+ 
a^2 b^3 + 8 b^4 + 2 a b^4 + b^5.
$$
The threefold $X_{a,b}$ has the following automorphisms, of order five and two respectively:
$$
\alpha_X(\mathbf{x})\,:=\,(\zeta^2 u :\zeta^3 v:\zeta x:y:\zeta^4 z),\quad
\iota_X(\mathbf{x})\,:=\,(v:u:z:y:x)~,
$$
where $\mathbf{x}=(u:v:x:y:z)$. These automorphisms generate a dihedral subgroup $D_5$ of order $10$ in ${\rm Aut}(X_{a,b})$.

\end{lemma}

\begin{proof}
In case $a_3=0$, we observe that changing coordinates and coefficients as follows: 
$$
(x_0,x_1,x_2,x_3,x_4)\,\longmapsto (x_0,x_4,x_3,x_2,x_1)~,
$$
$$
\underline{a}\,=\,(a_1,a_2,a_3,a_4,a_5,a_6,a_7)\,\longmapsto 
\underline{a}'\,:=\,(a_1,a_3,a_2,a_6,a_4,a_5,a_7)~,
$$
maps $F_{\underline{a}}$ to $F_{\underline{a}'}$. Hence we may assume that
$a_3\neq 0$. Since $F_{\underline{a}}$ defines a smooth cubic, we can introduce new
coordinates as follows:
$$
x_0\,:=\,2y/a_3,\quad x_1\,:=\,x/a_5,\quad   x_4\,:=\,z/a_7~.
$$
Then the equation for $X_{\underline{a}}$ is:
$$
xx_2^2+2yx_2x_3+zx_3^2\;+\;2a'_6z^2x_2+ 2a_4'x^2x_3\;+\;
a_1'y^3 + a_2'xyz\,=\,0
$$
where $(a'_6,a'_4,a'_1,a'_2)=\left(\ds\frac{a_6}{2a^2_7},\frac{a_4}{2a^2_5},\frac{8a_1}{a^3_3},\frac{2a_2}{a_3a_5a_7}\right)$. 
Now observe that  $a_6'\neq 0$ by Lemma \ref{discX}. We substitute 
$x_2:=\lambda u$, $x_3:=\lambda v$ with $\lambda=a_6'$ and we divide by $\lambda^2$
to get, with constants $a_i''$,
$$
xu^2+2yuv+zv^2\;+\;2z^2u+ 2a_4''x^2v\;+\;
a_1''y^3 + a_2''xyz\,=\,0~, 
$$
where $(a''_4,a''_1,a''_2)=\left(\ds\frac{a_4a^2_7}{a^2_5a_6},\frac{32a_1a^4_7}{a^3_3a^2_6},
\frac{8a_2a^3_7}{a_3a_5a^2_6}\right)$.
Observe that $a_4''\neq 0$ by Lemma \ref{discX} again. Let $\mu\in \overline{k}$ so that  
\begin{equation}\label{mu}
\mu^{15}=\frac{1}{a''_4}~.
\end{equation} We substitute  $(x,y,z,u,v)\mapsto (\mu^8 x,\mu^5 y,\mu^2 z,\mu^{-4} u,\mu^{-1} v)$. 
Then we get the equation $F_{a,b}=0$ as in the lemma with the parameters   
\begin{equation}\label{ab}
(a,b)=\left(\frac{32a_1a^2_5a^2_7}{a^3_3a_4a_6},\,\frac{8a_2a_5a_7}{a_3a_4a_6}\right)~. 
\end{equation}
By Lemma \ref{discX}, $X_{a,b}$ is smooth if and only if $D(a,b)\not=0$.  

To see that the automorphisms generate a $D_5$ it suffices to observe that
$\iota_X\alpha_X\iota_X^{-1}=\alpha_X^{-1}$.
\end{proof}

\begin{remark}\label{specialcase1}If $a_2=a_3=0$ and $X_{\underline{a}}$ is smooth,
then, by Lemma \ref{discX}, 
$\Delta(\underline{a})=27 a^3_1 a^2_4 a^2_5 a^2_6 a^2_7\not=0$. 
By the similar argument as above we see that $X_{\underline{a}}$ is isomorphic 
over $\overline{k}$ to 
$$X_0:x^3_0+x^2_1x_3+x_1x^2_2+x_2x^2_4+x^2_3x_4=0.$$ 
\end{remark}

\subsection{The intermediate Jacobian}\label{iJ}
In this subsection we assume $k\subset \C$.  Let $X$ be a smooth cubic threefold 
over $k$. 
By abuse of notation we denote again by $X$ its base change to $\C$.  
The intermediate Jacobian of $X$ is the 
five dimensional (principally polarized) abelian variety (\cite[Definition 12.2]{V1})
$$
J(X)\,:=\,H^3(X,\CC)/(F^2H^3(X,\C)+ H^3(X,\ZZ))\,=\,H^{1,2}(X)/H^3(X,\ZZ)~,
$$
where we used that $H^{0,3}(X)=0$. 

To find an isogeny decomposition of 
the intermediate Jacobian $J(X)$ we first consider the action of $\alpha_X$.

\begin{lemma}\label{decij}
Let $X=X_{\underline{a}}$ be a smooth cubic threefold in Section \ref{stmo} 
with automorphism $\alpha_X$. Then the eigenvalues of  $\alpha_X^*:H^{1,2}(X)\rightarrow H^{1,2}(X)$
are $\zeta^k$, $k=0,1,2,3,4$, each with multiplicity one.
\end{lemma}

\begin{proof}
Using Griffiths' theory of residues one has a natural
isomorphism (\cite[Corollary 6.12]{V2}) $H^{1,2}(X)\cong R_F^4$,
where $R_F^4=(S/J_F)_4$ is the degree four part of the Jacobian ring, 
so $S=\CC[x_0,\ldots,x_4]$ and $J_F$ is the ideal generated by the 
partial derivatives of $F$. 
The eigenvalues do not depend on the choice of $F$, so we can take 
$F=x_0^3+\ldots+x_4^3$, the Fermat cubic, and for the order five automorphism 
we take the cyclic permutation of the variables.
Then $R_F^4$ has a basis $r_j:=x_0\cdots \hat{x}_j\cdots x_4$, 
with $j=0,\ldots,4$. Then $\alpha_X^*r_j=r_{j+1}$, where we put $r_5:=r_0$
and thus the eigenvalues are as stated in the lemma.
\end{proof}

\begin{prop}\label{finedecij}
The intermediate Jacobian $J(X)$
is isogenous to the product:
$$
J(X)\,\sim\,E\,\times\,B^2~,
$$
where $B$ is an abelian surface, moreover 
$\QQ(\sqrt{5})\subset {\rm End}_{\overline{k}}(B)\otimes_{\Z}\Q$.
\end{prop}

\begin{proof}
The kernel of the endomorphism $\alpha_5^*-[1]$, where $[n]$ denotes the multiplication by $n$ on $J(X)$, has a connected component which is an elliptic curve since $\alpha_5^*$ has the eigenvalue $1$ with multiplicity one.
A complementary abelian fourfold $A$ can be defined as the image of $\alpha_5^*-[1]$, 
this fourfold has a automorphism of order five induced by $\alpha_5^*$ and its eigenvalues on $T_0A$, the tangent space at the origin, are the $\zeta^j$, $j=1,2,3,4$.
So we must show that the abelian subvariety $A\subset J(X)$
is isogenous to a selfproduct $B$. For this we use 
that ${\iota}_X\tilde{\alpha_X}{\iota}^{-1}_X\,=\,{\alpha_X}^{-1}$,
which implies that ${\iota}^*_X$ maps the eigenspace of $\alpha_X^*$ with eigenvalue $\zeta^j$ to the eigenspace with eigenvalue $\zeta_j^{-1}$. In particular, ${\iota}^*_X$ has two eigenvalues $+1$ and two eigenvalues $-1$ on $T_0A$. Hence the kernels $B_\pm$
of the endomorphisms ${\iota}_X^*\pm [1]$ of $A$ are abelian surfaces, 
which are isogenous, $B_+\sim B_-$, since ${\alpha}_X^*$ does not preserve them,
so $A\sim B_+\times B_-\sim B_\pm^2$.

The endomorphism ${\alpha}_X^*+({\alpha}^{-1}_X)^*$ of $A$ commutes with ${\iota}_X^*$, hence it acts on the $B_\pm$. As 
${\alpha}_X^*\in End(A)$
satisfies the equation $x^4+x^3+x^2+x+1=0$ (this follows by considering the action on $T_0A$), it follows (divide by $x^2$) that
$(x+x^{-1})^2+(x+x^{-1})-1=0$, thus
${\alpha}_X^*+({\alpha}^{-1}_X)^*$
generates a subfield 
$\QQ(\sqrt{5})\subset {\rm End}_{\overline{k}}(B_\pm)\otimes_\Z\Q$. 
Hence if we put $B:=B_+$ the proposition follows.
\end{proof}

\

\subsection{An algebraic construction of $J(X)$}\label{aiJ} 
We recall an algebraic 
construction of $J(X)$ due to Bombieri-Swinnerton-Dyer \cite{b&s}, Murre \cite{M}, 
Altman-Kleiman \cite{ak} (see also \cite{acm})
which is powerful as it works over any field of characteristic away from 2. 
However we restrict ourself to the case of the characteristic zero.  

Let $k$ be a field as in Section 1.  
Let $X$ be a smooth cubic threefold over $k$ defined by a cubic equation $F$ and 
let $S$ be the Hilbert scheme
of lines of $X$, which is a smooth surface over $k$.  
Then by the general theory of Fano threefolds
(cf.\cite{b&s},\cite{manin}), the Grothendieck motive $M:= h_3(X)$ associated to $X$ over $k$ coincides with the
motives $h_1(A)(\mathbb{L})$ over $k$ associated to the Albanese variety $A$ of $S$ where $\mathbb{L}$ stands for 
the Lefschetz motive. Note that $A$ is also defined over $k$. 
This $A$ is nothing but an algebraic substitute of $J(X)$. 
In fact, if $k\subset \C$, then $A_\C$ is isogenous to $J(X)$ 
by taking realizations of $h_3(X)=h_1(A)(\mathbb{L})$:
$$
A_\C\simeq H^1(A_\C,\C)/F^1H^1(A_\C,\C)+H^1(A,\Z)\sim H^3(X,\C)/F^2H^3(X,\C)+H^3(X,\Z)=J(X).
$$
We can not conclude that this map is an isomorphism since the equality  
$h_3(X)=h_1(A)(\mathbb{L})$ 
holds in the category of Grothendieck motives with the coefficients in $\Q$. 

By taking algebraic de Rham cohomology one has 
a functorial isomorphism 
$$H^3_{{\rm dR}}(X)\simeq H^1_{{\rm dR}}(A)$$
between $k$-vector spaces. 
This isomorphism also preserves the Hodge filtrations of both sides. 

We now describe the cohomology of the LHS. Notice that
$F^2H^3_{{\rm dR}}(X)$ `replaces' $H^{1,2}(X)$ in this section. 

We denote by $R=k[x_0,\ldots,x_4]$ the polynomial ring over $k$ 
with five variables and $R_d$ is the $k$-vector space of all homogeneous 
polynomial of degree $d$ $(\in \Z_{\ge 0})$. 
Consider $U:=\mathbb{P}^4\setminus X$. Then $U$ is an affine variety 
over $k$ with coordinate ring $\Gamma(U,\mathcal{O}_U)$, 
which consists of the homogeneous elements of degree 0 in 
$R[\frac{1}{F}]$. 
By excision for $(X,U)$  (cf. Theorem (3.3), p.40 and Proposition (3.4), p.41 of \cite{h}), there exists an isomorphism 
$H^3_{{\rm dR}}(X)\simeq H^4_{{\rm dR}}(U)$ as $k$-vector spaces. 
Furthermore, this isomorphism commutes with the action of 
${\rm Aut}(X)\cap {\rm Aut}(\mathbb{P}^4)$. 

Since $U$ is affine, 
$H^4_{{\rm dR}}(U)={\rm Ker}(d:\Omega^4_U\longrightarrow \Omega^5_U=0)/
{\rm Im}(d:\Omega^3_U\longrightarrow \Omega^4_U)$. 
Furthermore right hand side can be written as 
$$
\Big\{\frac{A\Omega}{F^i}\Big|\ A\in R_{3i-5},\ i=2,3,\cdots  \Big\}\Big/
\Big\{\partial_j\Big(\frac{A}{F^i}\Big)\Omega \Big|\ 
j=0,1,2,3,4,\ A\in R_{3i-4},\ i=2,3,
\cdots  \Big\}~,
$$
where $\Omega=\sum_{i=0}^4(-1)^i x_i dx_0\wedge \cdots \wedge 
\widehat{dx_i} \wedge \cdots dx_4$. 
 The subspace
$F^2H^3_{{\rm dR}}(X)$ corresponds to 
the image of 
$\Big\{\ds\frac{A\Omega}{F^{i+1}}\Big|\ A\in R_{3i-2}, i=1,2,\cdots  \Big\}$ 
in $H^4_{{\rm dR}}(U)$ and the
$\displaystyle\frac{x_i \Omega}{F^2}$, $i=0,\ldots,4$ 
give a basis of $F^2H^3_{{\rm dR}}(X)$.  

We now assume that $X$ admits an automorphism $\alpha_X$ of order five and $k$ contains $\zeta$.  
By changing the coordinates if necessary, we may assume $\alpha_X$ is the automorphism of $\mathbb{P}^4$ given 
at the beginning of Section \ref{stmo}. Then it is easy to see that 
$$v_j=\frac{x_j\Omega}
{F^2},\ j=0,1,2,3,4$$
are eigenvectors with eigenvalues $\zeta^j$ with respect to $\alpha_X$. 
By using this we can recover Proposition \ref{finedecij} for $A$:
\begin{prop}\label{general}
Let $A$ be the abelian five-fold as above. Then $A$ is isogenous over $k$ to 
the product of an elliptic curve $E$ over $k$ and the self-product of an abelian surface $B$ over $k$ so that 
$\Q(\sqrt{5})\subset {\rm End}_k(B)\otimes_\Z\Q$:
$$A\stackrel{k}{\sim}E\times B^2.$$
\end{prop}
\begin{proof}The proof is completely similar to Proposition \ref{finedecij} and is
therefore omitted. 
\end{proof}
Henceforth $J(X)$ stands for the intermediate Jacobian for $X$ in Section \ref{iJ} if 
$k\subset \C$ and the Albanese variety $A$ associated to $X$ in Section \ref{aiJ} otherwise. 
In both case we say that $J(X)$ is the intermediate Jacobian of $X$.

\begin{prop}\label{specialcase2}The intermediate Jacobian of the smooth cubic threefold 
$$X_0:\quad x^3_0+x^2_1x_3+x_1x^2_2+x_2x^2_4+x^2_3x_4=0$$
in Remark \ref{specialcase1} is isogenous to the five-fold product of an elliptic curve:
$$J(X_0)\,\stackrel{\overline{k}}{\sim}\,E^5_0,\qquad E_0:\quad y^2=x^3+1~.  $$
\end{prop}
\begin{proof}By Proposition \ref{finedecij} one has $J(X_0)\stackrel{k}{\sim} E_0\times B^2_0$.  
Let $\zeta_3$ be a primitive third root of unity. 
The action $\zeta_3:(x_0:x_1:x_2:x_3:x_4)\mapsto (\zeta_3x_0:x_1:x_2:x_3:x_4)$ induces 
the multiplication by $\zeta^2_3$ (resp. $\zeta_3$) on $F^1H^1_{{\rm dR}}(E_{0,\overline{k}})$ 
(resp. $F^1H^1_{{\rm dR}}(B_{0,\overline{k}})$). 
Here we used  the explicit basis of Section \ref{aiJ}.  
This means that $E_0,B_0$ have complex multiplication by $\Q(\sqrt{-3})$ and in particular $E_0$ has an affine model $y^2=x^3+1$ over $\overline{k}$ (up to isogeny). 
It follows from $\Q(\sqrt{5})\subset {\rm End}_k(B)$ that $B$ has non-simple CM by $\Q(\sqrt{5},\sqrt{-3})$ 
(cf.\ \cite[p.40, Proposition 3, 4; p.73, Example 8.4(2)]{st}). 
Therefore we have that $B_0\stackrel{\overline{k}}{\sim} E^2$ for some CM elliptic curve $E$. Recall the action of $\zeta_3$ on 
$F^1H^1_{{\rm dR}}(B_{0,\overline{k}})=\ds F^1H^1_{{\rm dR}}(E_{\overline{k}})^{\oplus 2}$
is multiplication by $\zeta_3$. 
This shows that $E\stackrel{\overline{k}}{\sim}E_0$.  
\end{proof}

\section{The curve $\cH_{a,b}$}

\subsection{Outline}
We use the simple equation $F_{a,b}=0$ for a smooth cubic threefold with an automorphism of order five to find a line $l$ 
in $X$ that is invariant under the $D_5$-action.
We explicitly construct the curve $\cH_l$ that parametrizes the lines in $X$ meeting $l$
and its involution $\iota_l$ following \cite{M}. 
The quotient $H_l:=\cH_l/\iota_l$ is a plane quintic curve and
the associated Prym variety $\mbox{Prym}(\cH_l/H_l)$
is the intermediate Jacobian $J(X)$ of $X$.

\subsection{The conic bundle}
Let $X_{a,b}$ be a smooth cubic threefold as in Lemma \ref{lemmaxab}. 
It has the equation, in $\PP^4$ with coordinates $x,y,z,u,v$,
$$
F_{a,b}:=\quad l_1u^2+2l_2uv+l_3v^2\;+\;2Q_1u+ 2Q_2v\;+\;
C~=0~,
$$
which is of the same form as (16) in \cite{M}, with $l_i,Q_j,C$ homogeneous of degree $1,2,3$ in $x,y,z$ respectively:
$$
l_1=x,\quad
l_2=y,\quad
l_3=z,\quad
Q_1=z^2,\quad
Q_2=x^2,\quad
C=ay^3+bxyz~.
$$

As in \cite[Section 1C]{M} we define a line $l\subset X_{a,b}$:
$$
l:\quad x\,=\,y\,=\,z\,=\,0\qquad(\subset\,X_{a,b})~.
$$
Notice that $l$ is invariant under the $D_5$-action. (There is only one other line
in $X_{a,b}$,
defined by $u=v=y=0$, which is also $D_5$-invariant.)

For a point $T:=(x:y:z)\in\PP^2$ we define a 2-plane $L_T$ in $\PP^4$:
$$
L_T\,:=\,\mbox{span}(l,T);\qquad p\,=\,(xt:yt:zt:u:v)\quad(\in\,L_T)
$$
provides a parametrization of this 2-plane with coordinates $(t:u:v)$ and 
$l$ is defined by $t=0$ in $L_T$.
The intersection of $L_T$ with $X$ is the union of the line $l$ and a conic $K_T$
(\cite[(24)]{M}):
$$
K_T\,:\qquad xu^2+2yuv+zv^2\;+\;2z^2tu+ 2x^2tv\;+\;
t^2(ay^3 + bxyz)=0.
$$
The conic $K_T$ degenerates if the point $T$ is on the quintic curve $H_{a,b}\subset\PP^2$ defined by
$$
\det\begin{pmatrix} x&y&z^2\\y&z&x^2\\z^2&x^2&ay^3+bxyz\end{pmatrix}\;=\,
x^5 - (b + 2)x^2yz^2 - (a - b)xy^3z + ay^5 + z^5=0.
$$
The automorphisms $\alpha_X$, $\iota_X$ of $X_{a,b}$ induce the automorphisms of order five and two respectively on $H_{a,b}$:
$$
\alpha_5,\,\iota:\;H_{a,b}\,\longrightarrow\,H_{a,b},\qquad
\alpha_5((x:y:z))\,=\,(\zeta x: y:\zeta^{-1}z),\quad
\iota((x:y:z))\,=\,(z:y:x)~.
$$
These automorphisms generate of subgroup isomorphic to the 
dihedral group $D_5$ of order ten in ${\rm Aut}(H_{a,b})$.
Similar to the proof of Lemma \ref{discX}, we found that the curve 
$H_{a,b}$ is smooth (and thus has genus $6$) if and only if $D_1(a,b)\neq 0$ where: 
\begin{equation}\label{disc3}
D_1(a,b):=(-1 + a + b) D(a,b).
\end{equation}

\subsection{Remark}\label{ab1}
Notice that the point $T=(1:1:1)$ lies on the quintic curve $H_{a,b}$ for all $a,b$.
The conic $K_T$ is defined by $u^2 + 2uv + 2ut + v^2 + 2vt + (a + b)t^2=0$
which, after substituting $u:=w-v-t$, becomes 
$w^2 + (a + b - 1)t^2=0$. Hence $K_T$ is a double line in case $a+b-1=0$, 
which is the reason it appears in $D_1(a,b)$ in equation (\ref{disc3}).

\subsection{Remark}
If we work over $\Q$ and take $a=b=-2$ 
then $H_{a,b}$ is the (twisted) Fermat quintic curve. 
It has good reduction outside 5 and 2, matching with $D_1(-2,-2)=5^4\cdot 2^4$.   
It is interesting to study the difference of the arithmetic conductor (that is $2^2\cdot 5^2$)  
of the curve  $H_{a,b}$ and our discriminant $D_1(a,b)$.

\subsection{The double cover}
If the quintic curve $H_{a,b}$ that parametrizes the degenerate conics 
is smooth, it has an etale double cover $\cH_{a,b}=\cH(l)_{a,b}$ 
that parametrizes the two irreducible components of the degenerate conic (see
\cite[1.24]{M}).

\begin{prop} If the curve $H_{a,b}$ is smooth 
(so $D_1(a,b)\neq 0$), 
then the double cover
$$
\pi_l:\,\cH_{a,b}\,\longrightarrow\,H_{a,b}
$$
is etale and $\cH_{a,b}$ has genus $11$.
The intermediate Jacobian $J(X)$ of $X$ is isomorphic (as principally polarized abelian variety) to the Prym variety $P_{a,b}=P_{a,b}(l)$ of this double cover:
$$
J(X)\,\cong\,P_{a,b}\,=\,{\rm Prym}(\cH_{a,b}/H_{a,b})~.
$$
Moreover, the double cover $\pi_l$ corresponds to the quadratic extension
$$
k(\cH_{a,b})\,=\, k(H_{a,b})(w),\qquad w^2\,=\,1-xz/y^2~,
$$
of the function field of $H_{a,b}$.
\end{prop}

\begin{proof}
See \cite{M}, we follow it closely in order to obtain the explicit expressions.

For a point $T\in H_{a,b}$, the conic $K_T$ has two irreducible components that 
are lines. For general $T$ the points of intersection of these two lines with $l$ 
will be distinct. The etale double cover $\cH_{a,b}\rightarrow H_{a,b}$ is 
thus defined by the degree two extension of the function field of $H_{a,b}$ 
which corresponds to the square root of discriminant of the quadratric polynomial 
obtained by intersecting $K_T$ with $l$.

The line $l$ is defined by $t=0$, hence $K_T\cap l$ is defined by 
$xu^2+2yuv+zv^2=0$
and this homogeneous quadratic polynomial in the coordinates $u,v$ on $l$ has discriminant 
$\Delta=y^2-xz$. 
Notice that $\Delta=0$ defines a conic in the plane of $H_{a,b}$ and 
that the intersection of $\Delta=0$ with $H_{a,b}$ consists of points with $z\neq 0$, 
putting $z=1$ we find these points by substituting $x=y^2$ into the equation for $H_{a,b}$:
$$
(y^2)^{5}\,-\,(b+2)(y^2)^2y\,-\,(a-b)y^2y^3\,+\,ay^5\,+\,1\,=\,
y^{10}\,-\,2y^5\,+\,1\,=\,(y^5\,-\,1)^2~.
$$
Thus the conic $\Delta=0$ is tangent to $H_{a,b}$ in all five intersection points, as it should be for the double cover it defines to be etale. 
A singular model of $\cH\subset\PP^3$ can be obtained simply by taking the inverse image of $H_{a,b}$ in the quadric surface defined by $w^2=y^2-xz$ (where $(x:y:z:w)$ are the homogeneous coordinates on $\PP^3$). 
That is, the singular model is a complete intersection of bidegree $(5,2)$ in $\PP^3$.

The function field of $H_{a,b}$ is generated by the rational functions $x/y,z/y$ and 
as we observed, the function field of $\cH_{a,b}$ is generated by the square root of $(y^2-xz)/y^2={1-(xz/y^2)}$.
\end{proof}

\subsection{Automorphisms of $\cH_{a,b}$}\label{autch}
From the explicit construction of the curve $\cH_{a,b}$, 
in particular the fact that the rational function $1-xz/y^2\in k(H_{a,b})$
is invariant under the automorphisms $\alpha_5,\iota$ of $H_{a,b}$,
we see that the automorphism $\alpha_X$, $\iota_X$ of the threefold $X_{a,b}$
induce automorphisms $\tilde{\alpha}_5$, $\tilde{\iota}$ on $\cH_{a,b}$
which generate a $D_5\subset {\rm Aut}(\cH_{a,b})$.
The covering involution of the double cover $\cH_{a,b}\rightarrow H_{a,b}$ will be denoted by $\iota_l$ ($\in {\rm Aut}(\cH_{a,b})$).

The action of these automorphisms on the
rational functions $x/y,z/y,w\in k(\cH_{a,b})$
is 
{\renewcommand{\arraystretch}{1.3}
$$
\begin{array}{rlcl}
\iota_l:&(\,x/y,\,z/y,w)&\longmapsto&(\,x/y,\,z/y,\,-w),\\
\tilde{\alpha}_5:&(\,x/y,\,z/y,\,w)&\longmapsto&(\,\zeta x/y,\,\zeta^{-1}z/y,\,w),\\
\tilde{\iota}:&(\,x/y,\,z/y,w)&\longmapsto&(\,z/y,\,x/y,\,w)~.
\end{array}
$$
}
Notice that $\iota_l$ commutes with both $\tilde{\alpha}$ and $\tilde{\iota}$,
hence $\tilde{\alpha}$ and $\tilde{\iota}$ 
induce automorphisms of the Prym variety $Prym(\cH_{a,b}/H_{a,b})$.

Since the line $l$ is fixed by $\iota_X$, this involution on $X_{a,b}$ induces an involution on the curve $\cH_{a,b}$ which parametrizes lines in $X_{a,b}$ meeting $l$. The following lemma identifies this involution.

\begin{lemma} \label{invcH}
The involution $\iota_X$ on the threefold $X_{a,b}$ induces the involution
$\tilde{\iota}$ on $\cH_{a,b}$.

The quotient curves $\cH_{a,b}/\tilde{\iota}$ and 
$\cH_{a,b}/\tilde{\iota}\iota_l$ have genus $4$ and $6$ respectively.
\end{lemma}

\begin{proof}
The involution $\iota_X$ induces $\iota$ on $H_{a,b}$
with $\iota(x:y:z)=(z:y:z)$. 
As $\cH_{a,b}$ is defined by the quadratic equation $w^2=1-xz/y^2$,
and $\iota_l$ changes only the sign of $w$, we must have that $\iota_X$ induces
either $\tilde{\iota}$ or $\tilde{\iota}\iota_l$ on $\cH_{a,b}$. To find out which, we consider the fixed points of these involutions on $\cH_{a,b}$.

The fixed points of $\tilde{\iota}$ and $\tilde{\iota}\iota_l$ 
map under $\pi_l$ to the fixed points of $\iota$ on $H_{a,b}$. 
These are the point $(1:0:-1)$ and the five points of intersection of $H_{a,b}$ with the line $x=z$, one of which is $(1:1:1)$.  In the other four points, the function $1-xz/y^2$ has a non-zero value and thus the 8 points on $\tilde{H}_{a,b}$ mapping to them are fixed points of $\tilde{\iota}$, and thus none of these points is a fixed point of $\tilde{\iota}\iota_l$.

It is not hard to check that the corresponding 4 pairs of lines on $X_{a,b}$ are
interchanged by $\iota_X$, hence $\iota_X$ induces $\tilde{\iota}$ on $\cH_{a,b}$.

The fixed points of $\tilde{\iota}$ on $\cH_{a,b}$ are thus among the four points which map to $(1:0:-1)$ and $(1:1:1)$ and 
they correspond to the lines meeting $l$ which are fixed by $\iota_X$.
For $T=(1:0:-1)$ these lines form the conic $K_T$ which is 
$(u - v + 2t)(u+v)=0$. The line on $X_{a,b}$ corresponding to $u+v=0$ is the line parametrized by $(u:v:x:y:z)=(s:-s:t:0:-t)$,
and this line is mapped into itself by $\iota_X$. Thus both of the points on $\cH_{a,b}$ mapping to $T$ are fixed under $\tilde{\iota}$. Similarly,
for $(1:1:1)\in H_{a,b}$,  the two corresponding
lines on $X_{a,b}$ are parametrized by $(u:v:x:y:z)=
(s+ct:-s+ct:t:t:t)$, with $4c^2+4c+a+b=0$, and thus both are fixed by $\iota_X$. 
Therefore $\tilde{\iota}$ has $4$ fixed points 
and $\tilde{\iota}\iota_l$ has $8$ fixed points on $\cH_{a,b}$. 

The genera of the quotient curves now follow from the Hurwitz formula.
\end{proof}

\

\section{Quotients of $\cH_{a,b}$}

\subsection{Outline}
In this section we determine the isogeny classes of the factors $E_{a,b}$ and $B_{a,b}$ of 
$J(X_{a,b})$ explicitly. 
Throughout this section we assume that $X_{a,b}$ is smooth, 
hence $D(a,b)\not=0$. 

\subsection{Quotient maps of degree five}\label{qm5}
We determine the quotient curves $\overline{H}_{a,b}$ and ${\overline{\cH}}_{a,b}$ of $H_{a,b}$ and $\cH_{a,b}$ 
by the automorphisms of order five.
In Proposition \ref{elliptic} we will show that
these quotient curves have genus two and three respectively,
thus these quotient maps are etale. 
Moreover, the Prym variety of the double cover $\overline{\cH}_{a,b}\rightarrow \overline{H}_{a,b}$ is thus an elliptic curve $E_{a,b}'$ which we determine explicitly, it is in fact a quotient of 
${\overline{\cH}}_{a,b}$. All these curves and quotient maps fit in the following diagram:
$$
\begin{CD}
{{\cH}}_{a,b}  @> 5:1 >>  \overline{\cH}_{a,b}:=\cH_{a,b}/\tilde{\alpha}_5
@>2:1>>E_{a,b}'\\
  @V 2:1 VV @V 2:1 VV @.\\
 {H}_{a,b}  @>5:1 >> \overline{H}_{a,b}\,:=\,H_{a,b}/\alpha_5@. 
\end{CD}
$$

\begin{prop}\label{elliptic} 
The curves in the diagram above have defining equations:
{\renewcommand{\arraystretch}{1.5}
$$
\begin{array}{rcl}
{\overline{\cH}}_{a,b}:&\quad& r^2=f_4(1-w^2)~,
\\
{\overline{H}}_{a,b}:&& s^2\,=\,(t-1)f_4(t)~,
\\
E_{a,b}':&&r^2=f_4(1-w)~,
\end{array}
$$
}
where the degree four polynomial $f_4$ is defined by:
$$
f_4(t)\,=\,
t^4 \,-\, \Big(\frac{b^2}{4} + b\Big)t^3\, -\, 
\Big(\frac{ab}{2} + a - \frac{b^2}{4}\Big)t^2 \,-\, 
\Big(\frac{a^2}{4} -  \frac{ab}{4}\Big)t \,+\, \frac{a^2}{4}~.
$$

A Weierstrass model for $E_{a,b}'$ and the maps between the curves are given in the proof.
\end{prop}

\begin{proof}
We introduce rational functions $s,t$ on $H_{a,b}$ which are invariant under the automorphism $\alpha_5:(x:y:z)\mapsto (\zeta x:y:\zeta^{-1} z)$ of order five:
$$
s\,:=\,(x/y)^5,\quad t:=\,xz/y^2,\qquad\mbox{then} \quad t^5/s\,=\,(z/y)^5~. 
$$
Now we divide the defining polynomial for $H_{a,b}$ by $y^5$ 
and rewrite it with these invariant functions:
$$
(x/y)^5 - (b + 2)(xz/y^2)^2 - (a - b)xz/y^2 + a + (z/y)^5\,=\,
s-(b+2)t^2-(a-b)t+a+t^5/s~,
$$
Multiplying by $s$ we obtain a polynomial of degree two in $s$:
$$
s^2\,-\,\big((b+2)t^2+(a-b)t-a\big)s\,+\,t^5=0,
$$
which defines a genus two curve $\overline{H}_{a,b}$ which is thus the quotient of $H_{a,b}$ by $\alpha_5$.

Notice that $\overline{H}_{a,b}$ is a double cover of $\PP^1$, with coordinate $t$, branched in the six points where the discriminant is zero. 
It is also easy to find the Weierstrass form by substituting $s:=s+\big((b+2)t^2+(a-b)t-a\big)/2$, one finds
$
\overline{H}_{a,b}$: $s^2=(t-1)f_4(t),
$
with $f_4$ as above.

The function $1-xz/y^2\in k(H_{a,b})$, whose square root $w$ defines $\cH_{a,b}$,
can be written as
\begin{equation}\label{symbol}
w^2\,=\,1\,-\,xz/y^2\,=\,1\,-\,t\qquad \mbox{so}\quad t\,=\,1-w^2
\end{equation}
and substituting this in the Weierstrass equation of $\overline{H}_{a,b}$ and defining $r:=s/w$
we get the Weierstrass equation for the double cover ${\overline{\cH}}_{a,b}$
of $\overline{H}_{a,b}$:
$$
{\overline{\cH}}_{a,b}:\quad r^2\,=\,f_4(1-w^2)~. 
$$
In particular, ${\overline{\cH}}_{a,b}$ is hyperelliptic of genus three
and thus the map ${\overline{\cH}}_{a,b}\rightarrow \overline{H}_{a,b}$
is an etale double cover. The Prym variety of this double cover can be determined as
in \cite{Mpv}.
The discriminant of ${\overline{\cH}}_{a,b}$ is given by 
$$
2^{-8}(1-a-b)a^{10}\Delta(a,b).
$$ 
This curve has a subgroup isomorphic to $(\ZZ/2\ZZ)^2\subset {\rm Aut}({\overline{\cH}}_{a,b})$;
there is the hyperelliptic involution $(s',w)\mapsto (-s',w)$, the covering involution $(s',w)\mapsto (-s',-w)$ 
and their product is the involution $(s',w)\mapsto (s',-w)$ with quotient a genus one curve 
$$
C_{a,b}:\quad s'^2\,=\,f_4(1-w)~.
$$
By Proposition 1.2.1 of \cite{cornel}, 
one has the following cubic Weierstrass model which is a 
smooth birational model of $C_{a,b}$:
{\renewcommand{\arraystretch}{2.0}
\begin{equation}
\begin{array}{c}
E''_{a,b}:\quad y^2+a_1xy+a_3y\,=\,x^3+a_2x^2+a_4 x+a_6~,\\
 (x,y)=\Big(\ds\frac{2s'+(\frac{b^2}{4}+b-4)w+2}{w^2},
\frac{4s'+4a_2w^2+(\frac{b^2}{2}+2b-8)w+4}{w^3}\Big)~,
\end{array}
\end{equation}
}
with 
$$
a_1=\frac{b^2}{4}+b-4,\ a_3=\frac{(a+b)^2}{2}+4a+6b-8,\ 
a_2=2-a-\frac{ab}{2}-\frac{b^4}{64}-\frac{b^3}{8}-\frac{b^2}{4}-b,
$$
$$
a_4=4 ( a + b-1),\ a_6=(a + b-1)(8-4a-2ab-\frac{b^4}{16}-\frac{b^3}{2}-b^2-4b).
$$
The discriminant and the $j$-invariant of $E''_{a,b}$ are given by 
$-2^{-4}a^5\Delta(a,b)$ and 
$$
j(E''_{a,b})\,=\,
-\ds\frac{16(64 a^2 + 4 a^2 b + 16 a b^2 + a^2 b^2 + 2 a b^3 + b^4)^3}{a^5\Delta(a,b)}
$$ 
respectively. Notice that the factor $(a+b-1)$ disappears. 
This elliptic curve  is actually the Prym variety ${\rm Prym}({\overline{\cH}}_{a,b}/\overline{H}_{a,b})$.
\end{proof}

\begin{prop}\label{geoE}
If the threefold $X_{a,b}$ is smooth, 
the elliptic curve $E_{a,b}''$ is isogenous to the elliptic factor  $E_{a,b}$ of $J(X_{a,b})$
which we found in Proposition \ref{finedecij} (and also in Proposition \ref{general}).
\end{prop}
\begin{proof} In case $H_{a,b}$ is smooth, 
it follows from the diagram in Section \ref{qm5} that $E_{a,b}''$
is isogenous to an abelian subvariety of 
$J(X)={\rm Prym}(\cH_{a,b}/H_{a,b})$ 
and that
the tangent space in $0$ to this subvariety is the eigenspace of $\tilde{\alpha}_5^*$
acting on $T_0J(X)$ with eigenvalue $1$. 
Thus $E_{a,b}''$ is isogenous to $E_{a,b}$ if $(1-a-b)D(a,b)\neq 0$. 

What is left is the case $1-a-b=0$ and 
$D(a,b)=a\Delta(a,b)\neq 0$. 
Let $S$ be the affine open subscheme of $\mathbb{A}^2_{a,b}$ defined by $D(a,b)\not=0$. 
Let $P_0=(a_0,b_0)$ be a (geometric) point of $S$ with $1-a_0-b_0=0$. 
Take a line $Z$ on $S$ passing $P_0$ but different from the line $1-a-b=0$. 
Let $R$ be the localization of $\mathcal{O}_Z$ at $P_0$ and $Q(R)$ its field of fractions. 
We view $E_{a,b}'', E_{a,b}$ as smooth families over $Z$.  
Let us consider their base change to Spec$R$ and  
take the N\'eron models $\mathcal{E}_{a,b}'$, $\mathcal{E}_{a,b}$ over Spec$R$ respectively. 
What we have shown above is that ${\rm Hom}_{Q(R)}(E_{a,b}', E_{a,b})\neq 0$. As
${\rm Hom}_{Q(R)}(E_{a,b}', E_{a,b})={\rm Hom}_{R}(\mathcal{E}_{a,b}',\mathcal{E}_{a,b})$ 
(cf. p.12, Definition 1 and p.16, Corollary 2 of \cite{blr}) 
and the specialization map 
$$
{\rm Hom}_{R}(\mathcal{E}_{a,b}',\mathcal{E}_{a,b})\lra 
{\rm Hom}_{k}(\mathcal{E}_{a_0,b_0}',\mathcal{E}_{a_0,b_0})
$$ 
is injective (cf. \cite[p.45, Theorem 3.2]{l}), we find that $E''_{a_0,b_0}=\mathcal{E}_{a_0,b_0}'$ and 
$E_{a_0,b_0}=\mathcal{E}_{a_0,b_0}$ are isogenous. 
Hence we conclude that  $E_{a,b}''$ is isogenous to $E_{a,b}$ whenever
$X_{a,b}$ is smooth.
\end{proof}

\subsection{Quotient maps of degree two} \label{aq}
In the previous section we studied quotients of the curve $\cH_{a,b}$ by  the automorphism of order five. 
This curve also has the involution
$\tilde{\iota}$ (cf.\ Section \ref{autch}) which 
is a lift of the involution $\iota$ on $H_{a,b}$. 
Thus we have a commutative diagram:
$$
\begin{CD}
 \cH_{a,b} @> 2:1 >>   \mathcal{D}_{a,b}\,:=\,\cH_{a,b}/\tilde{\iota}\\
 @V 2:1 VV @V 2:1 VV \\
H_{a,b}  @>2:1 >>  D_{a,b}\,:=\, H_{a,b}/\iota~.
\end{CD}
$$

\begin{prop}
The curves $\mathcal{D}_{a,b}$ and $D_{a,b}$ have genus four and two respectively. 

The quotient curve $\mathcal{D}_{a,b}=\cH_{a,b}/\tilde{\iota}$ has a (singular) 
projective model in $\PP^2$ defined by 
\begin{equation}\label{cdab}
(T_1-2T_2)(T^2_1+T_1T_2-T_2^2)^2
\,+\,
((a+b+4)T_2^3-10T_1T_2^2+5T^3_1)w^2
\,+\,
((-2-b)T_2+5T_1)w^4=0.
\end{equation}
The genus two curve $D_{a,b}=H_{a,b}/\tilde{\iota}$ has the Weierstrass model
\begin{equation}
D_{a,b}: \quad v^2_2\,=\,5T^6_1+(4b+8)T^5_1+10(a-b)T^3_1-20aT_1+(a+b)^2+8a.
\end{equation}
The discriminant of $D_{a,b}$ is $$2^{12}5^5(a+b-1)^4\Delta(a,b).$$
\end{prop}

\begin{proof}
The genus of the quotient curves is given in Lemma \ref{invcH}.
We introduce the following $\tilde{\iota}$-invariant rational functions on $\cH_{a,b}$: 
$$
T_1\,:=\,x+y,\quad T_2\,:=\,xz\,=\,y^2(1-w^2)~,
$$
where $w$ is as in equation (\ref{symbol}). 
If we dehomogenize by putting $y=1$, we get a rational map from $\cH_{a,b}$ to the singular quintic in the proposition. Since this quintic curve has two nodes in the points $w=T^2_1+T_1-1=0$, its genus is $\frac{(5-1)(5-2)}{2}-2=4$ and thus
it is birational to the quotient curve $\mathcal{D}_{a,b}$.

The covering involution $\cH_{a,b}\rightarrow H_{a,b}$ induces
the involution $\tau:(T_1,w)\lra (T_1,-w)$ on the curve $\mathcal{D}_{a,b}$.
The fixed points of $\tau$ in the singular model are $(0:0:1)$ and $(2:1:0)$
(in fact, the nodes are also fixed points, but as the two tangent lines in each node are interchanged, there are no fixed points on the smooth model mapping to the nodes).
Thus the genus of the quotient curve will be equal to two.

If we put $v:=w^2$, the quotient curve $D_{a,b}:=\mathcal{D}_{a,b}/\tau$ is given by 
$$
(T_1-2)(T^2_1+T_1-1)^2+(a+b+4-10T_1+5T^3_1)v+(-2-b+5T_1)v^2=0.
$$
To find a Weierstrass model of this genus two curve, 
we replace $v$ with $\ds\frac{v_1}{2+b-5T_1}$ and next replace $v_1$ by $\ds\frac{v_2}{2}+\frac{1}{2}(a+b+4-10T_1+5T^3_1)$ 
and we find the equation given in the proposition.
\end{proof}

\begin{prop} \label{geoB}
In the decomposition of the intermediate Jacobian $J(X_{a,b})$,
$$
J(X_{a,b})\,\stackrel{k}{\sim}\,E_{a,b}\,\times\,B_{a,b}^2
$$
(see Proposition \ref{finedecij}), the abelian surface 
$B_{a,b}$ is isogenous to the (two-dimensional) Prym variety of the double cover
$\cD_{a,b}\rightarrow D_{a,b}$,
$$
B_{a,b}\,\sim\,{\rm Prym}(\cD_{a,b}/D_{a,b})~.
$$
\end{prop}

\begin{proof}
From the diagram in Section \ref{aq} it follows that 
${\rm Prym}(\cD_{a,b}/D_{a,b})$ is isogenous to an abelian subvariety of $J(X_{a,b})
\cong{\rm Prym}(\cH_{a,b}/H_{a,b})$, 
and from the proof of Proposition \ref{finedecij} one finds that it
must be isogenous to $B_{a,b}$.
\end{proof}

\

\subsection{The Jacobian of a genus two curve} 
We now try to find  the Jacobian of a genus two curve
which is isogenous to the abelian surface $B_{a,b}$.
To do this we first consider, for general $a,b$, the normalization of the $4:1$ cover 
$$
\mathcal{D}_{a,b}\lra D_{a,b}\lra \mathbb{P}^1,
\qquad (T_1,w)\,\longmapsto\,(T_1,v_2)\,\longmapsto\,T_1
$$ 
where we put $T_2=1$.
So we view $k(\cD_{a,b})=k(T_1,w)$ as a quartic extension of $k(T_1)$ and we consider its Galois closure.

\

\begin{prop}\label{nor}
Assume that the map $\mathcal{D}_{a,b}\lra \mathbb{P}^1$ is not a Galois cover,
and let
$\widetilde{\mathcal{D}}_{a,b}\lra \mathbb{P}^1$ 
be its (Galois) normalization.  
Then the double cover 
$\widetilde{\mathcal{D}}_{a,b}\lra\mathcal{D}_{a,b}$ is defined by the
function field extension 
$$
k(\mathcal{\widetilde{D}}_{a,b})\,=\,k(\mathcal{D}_{a,b})({t}),
\qquad t^2=\ds\frac{T_1-2}{5T_1-b-2}~,
$$
where $k(\mathcal{D}_{a,b})$ is the function field of $\mathcal{D}_{a,b}$. 
\end{prop}

\begin{proof}
The equation of $\mathcal{D}_{a,b}$ can be written as $pw^4+qw^2+r$, with $p,q,r$ polynomials
in $T_1$. 
So $T_1,w\in k(\cD_{a,b})$, and $w$ is a root of the irreducible polynomial $pX^4+qX^2+r\in k(T_1)[X]$. Another root is obviously $-w$, and in 
$k(T_1,w)[X]=k(\cD_{a,b})[X]$ one finds the factorization:
$$
pX^4+qX^2+r\,=\,p(X^2\,-\,w^2)(X^2\,-\,s)~,\qquad 
s\,=\,\frac{r}{pw^2}\,=\,\frac{(T_1-2)(T_1^2+T_1-1)^2}{(5T_1-b-2)w^2}~.
$$
Thus in general we need to adjoint the square root of $s$, or equivalently $t$ as in the proposition, to obtain a Galois extension.
\end{proof}


\subsection{The Galois group}
For $a,b$ such that the 4:1 map $\mathcal{D}_{a,b}\rightarrow \PP^1$ is
not Galois (this would imply $b\not=8$), 
the Galois cover $\widetilde{\cD}_{a,b} \rightarrow \PP^1$
has group
$\mbox{Gal}(k(\tcD_{a,b})/k(T_1))\cong D_4$, the dihedral group of order eight.
Since the roots of the polynomial $pX^4+qX^2+r$ are $\pm w$ and 
$\pm w_1=\pm t(T_1^2+T_1-1)/w$,
this Galois group is generated by $\sigma$, 
which acts on the roots $w,w_1$ as $w_1,-w$ and 
the involution $\tau$ on $\widetilde{\cD}_{a,b}$,
which sends $t\mapsto -t$ and thus maps $w,w_1$ to $w,-w_1$.

In the following proposition we establish and use the following diagram, 
where each arrow is a double cover:
$$
\begin{array}{ccccl}
&&\widetilde{\mathcal{D}}_{a,b} &&\\
&\swarrow&&\searrow&\\
{\mathcal{D}}_{a,b}&&&&C_{a,b}:=\widetilde{\mathcal{D}}_{a,b}/\sigma\tau\\
\downarrow&&&&\,\downarrow\\
D_{a,b}&&&&
C_0:=\widetilde{\mathcal{D}}_{a,b}/\langle\sigma\tau,\sigma^2\rangle\cong\PP^1\\
&\searrow&&\swarrow&\\
&&\PP^1&&
\end{array}
$$

\begin{prop}\label{pan}Assume that the 4:1 map 
$\mathcal{D}_{a,b}\rightarrow \PP^1$ is not Galois.
Then the abelian surface $B_{a,b}$ in $J(X_{a,b})$, 
which is isogenous to ${\rm Prym}(\cD_{a,b}/D_{a,b})$, is also isogenous
to the Jacobian of the genus two curve $C_{a,b}:=\widetilde{\mathcal{D}}_{a,b}/\sigma\tau$. 
More precisely, ${\rm Prym}(\cD_{a,b}/D_{a,b})$ is isomorphic, 
as a principally polarized abelian variety, to ${\rm Jac}(C_{a,b})$.

The curve $C_{a,b}$ is defined by the Weierstrass equation
\begin{equation}\label{genus2model}
C_{a,b}\,:\quad s^2\,=\,-(b-8)f_6(t),\qquad f_6(t)\,:=\,\sum_{i=0}^6b_it^i,\quad\mbox{with}
\end{equation}
$$
\begin{array}{rcl}
b_0&=&-(a + b+24),\\
b_1&=&10(b-8), \\
b_2&=&5(3a + 13b - 8), \\
b_3&=&-10(b-8)(b+2),\\
b_4&=&-5(15a + 6b^2 + 19b - 56),\\
b_5&=&2(b-8)(b^2+9b-11),\\  
b_6&=&5(25a + b^3 + 6b^2 - 13b + 8)~.
\end{array}
$$
The discriminant of the curve $C_{a,b}$ is given by 
$$
2^{6}5^5(b-8)^{22}\Delta(a,b)^2~.
$$
\end{prop}

\begin{proof}
The involutions $\sigma\tau$ and $\sigma^2$ act as $(w,w_1)\mapsto (w_1,w)$ and $(w,w_1)\mapsto (-w_1,-w)$ respectively. In particular, both fix 
$t=ww_1/(T_1^2+T_1-1)$. Thus the function field of 
$C_0:=\widetilde{\mathcal{D}}_{a,b}/\langle\sigma\tau,\sigma^2\rangle$, which is a quadratic extension of $k(T_1)$, is the field $k(t)$ and hence $C_0\cong\PP^1$.

The function field of $C_{a,b}:=\widetilde{\mathcal{D}}_{a,b}/\sigma\tau$
is the extension of $k(T_1)$ generated by $w+w_1$ and $ww_1$ and is a quadratic extension of $k(t)$. From the factorization
$$
\begin{array}{rcl}
pX^4+qX^2+r&=&p(X\,-\,w)(X\,+\,w)(X\,-\,w_1)(X\,+\,w_1)\\
&=&p(X^2-(w+w_1)X+ww_1)(X^2+(w+w_1)X+ww_1)
\end{array}
$$
we find that $q/p=2ww_1-(w+w_1)^2$ and $r/p=(ww_1)^2$. Therefore
$$
(w+w_1)^2\,=\,2ww_1\,-\,q/p\,=\,2t(T_1^2+T_1-1)\,-\,
(5T_1^3-10T_1+(a+b+4))/(5T_1-(b+2))~.
$$
From the definition of $t$ we also have 
$T_1={((b+2)t^2-2)}/{(5t^2-1)}$, 
thus we can write $(w+w_1)^2$ as a rational function in $t$. 
Defining $s:=(b-8)(5t^2-1)(w+w_1)$ clears the denominator and we obtain the 
equation for $C_{a,b}$ as in the proposition.

For such a diagram defined by a $D_4$-cover, a general result of Pantazis \cite[Proposition 3.1]{Pa}
asserts that ${\rm Prym}(\cD_{a,b}/D_{a,b})$ and 
${\rm Prym}(C_{a,b}/C_0)={\rm Jac}(C_{a,b})$ are dual abelian varieties. 
As ${\rm Jac}(C_{a,b})$ is principally polarized, it is self dual, and the proposition is proved. 
\end{proof}

\subsection{The case $b=8$}
In this case, the extension $k(\mathcal{D}_{a,b})/k(T_1)$ is already normal,
see Proposition \ref{nor}. 
The affine model
$$
\mathcal{D}_{a,8}:\quad (T_1-2)(T^2_1+T_1-1)^2+(a+12-10T_1+5T^3_1)w^2+5(T_1-2)w^4=0
$$ 
admits the following two new involutions 
$$
\iota^{\pm}_8:\quad (T_1,w)\,\longmapsto\, 
\left(T_1,\,\pm\frac{T^2_1+T_1-1}{\sqrt{5}w}\right)~.
$$
Put 
$$
(X,Y)\,=\,
\Big(\frac{-1}{T_1-2},\,\frac{1}{T_1-2}(\sqrt{5}w\pm \frac{T^2_1+T_1-1}{w})\Big)~.
$$
The quotient curves $C^{\pm}_a=\mathcal{D}_{a,8}/\iota^{\pm}_8$ 
are elliptic curves with the following affine models:
$$C^{\pm}_{a}:\quad
Y^2=(a+32)X^3+10(-5\pm \sqrt{5})X^2-10(-3\pm\sqrt{5})X-5\pm2\sqrt{5}$$
with the discriminants $-5(2\mp\sqrt{5})^2(a+32)(27a+64)$ and 
the $j$-invariants (independent of the choice of sign): 
$$
j(C^{\pm}_{a})\,=\,\ds\frac{2^{14}\cdot 5^2(3a-4)^3}{(a+32)^3(27a+64)}~.
$$ 
Since we have assumed that $k$ contains $\zeta$, it is easy to see that the curves 
$C^{\pm}_{a}$ are isomorphic over $k(\sqrt{-1})$ to the elliptic curve with the  
following affine model 
$$
C_a:5y^2=  5 (a+32) x^3 - 100 x^2+ 20 x -1,\qquad 
(x,y)\,=\,\left(\frac{5\pm \sqrt{5}}{10}X,\,\frac{\sqrt{5\pm 2\sqrt{5}}}{5}Y\right)~,
$$
where  $\sqrt{5+2\sqrt{5}}=\sqrt{-1}(1+2\zeta^3+2\zeta^4)$ and $\sqrt{5-2\sqrt{5}}=-\sqrt{-1}(1+2\zeta+2\zeta^3)$. 
Note that $D(a,8)=a (a+32)^2 (27 a+64)$. 
Hence $C_a\simeq C^{\pm}_a$ is smooth provided if $X_{a,8}$ is smooth. 
 
\begin{prop}\label{nor1}Keep the notation as above. Assume that $b=8$. 
Then 
$$
B_{a,8}\,\stackrel{k}{\sim}\,C^+_a\times C^-_a\,
\stackrel{k(\sqrt{-1})}{\simeq}\,C^2_a~.
$$
\end{prop}

\begin{proof}
We denote by $\pi:\mathcal{D}_{a,8}\lra D_{a,8}=\mathcal{D}_{a,8}/\tau$ the quotient map and 
$\pi^{\pm}_a:\mathcal{D}_{a,8}\lra C^{\pm}_a$ as well. Then it is easy to see that $(\pi^{\pm}_a)_\ast\circ \pi^\ast=0$ and 
$(\pi^{-}_a)_\ast\circ (\pi^{+}_a)^\ast=0$ on Jacobians respectively.  The claim follows from this.  
\end{proof}

Summing up, 
we have proved the following: 
\begin{prop}Assume that $X_{a,b}$ is smooth. 
If the map $\mathcal{D}_{a,b}\rightarrow \PP^1$ is not Galois (hence $b\not=8$), then 
$$
J(X_{a,b})\,\stackrel{k}{\sim}\,E_{a,b}\times {\rm Jac}(C_{a,b})^2.
$$
In case $b=8$, this map is Galois, and moreover 
$$
J(X_{a,8})\,\stackrel{k}{\sim}\,
E_{a,8}\times (C^+_a)^2\times (C^-_a)^2\,
\stackrel{k(\sqrt{-1})}{\simeq}\,
E_{a,8}\times C_a^4~.
$$
\end{prop}

\subsection{Moduli}
The moduli space of principally polarized abelian surfaces with 
real multiplication by $\ZZ[(1+\sqrt{5})/2]$ was studied in \cite{wilson}.
A general such abelian surface is the Jacobian of a genus two curve and this curve
is determined by six points on $\PP^1$. In \cite[Section 5]{wilson} one finds that these six points satisfy the equation (for $m=6$)
$$
H_m(z)\,:=\,12z_m^4\,-\,4\tau_2(z)z_m^2\,+\,\tau_2^2(z)\,-\,4\tau_4(z)\,=\,0 
$$
where $\tau_k(z)=\sum z_j^k$ and $z_1,\ldots,z_j$ are certain (explicit) functions in the coordinates of the six points. 
These $z_j$ satisfy $\sum z_j=0$ and $\sum z_j^3=0$, so $z=(z_1:\ldots:z_6)$ is a point of the Segre cubic threefold, and permuting the points on $\PP^1$ 
permutes the $z_j$.
Moreover,  \cite[p.133]{wilson} gives explicit 
expressions of the Igusa invariants $I_2(z),I_4(z),I_6(z),I_{10}(z)$ 
of the genus two curve in terms of $z$. 

We computed the Igusa invariants $I_k(C_{a,b})$ of the curve $C_{a,b}$ with Magma \cite{magma}. We checked that there is a (weighted) homogeneous polynomial 
$P$ of degree $24$ in the Igusa invariants, with $16$ terms,
$$
P\,:=\,I_2^4 I_4^4\, -\, 12 I_2^3 I_4^3 I_6\,-\,972 I_2^3 I_4^2 I_{10}\,-\,\ldots
\,-\,15116544 I_4 I_{10}^2 + 81 I_6^4~,
$$
such that $P(C_{a,b})=0$ for all $a,b$. 
Next we parametrized the Segre threefold by taking six general points, $(1:x_1)$, $(1:x_2)$, $(1:x_3)$, $(1:1)$, $(1:0)$, $(0:1)$, in
$\PP^1$, and we computed the functions $z_1,\ldots,z_6$ as well as the Igusa invariants $I_k(z)$ explicitly. 
Substituting the $I_k(z)$ in $P$ and factorizing the result, we found that
$$
P(z)\,=\,2^{-28}3^6\,H_1(z)H_2(z)H_3(z)H_4(z)H_5(z)H_6(z)~.
$$ 
Thus indeed
$\QQ(\sqrt{5})\subset {\rm End}_{\overline{k}}(J(C_{a,b}))\otimes_\Z\Q$,
as we already know from our construction. 
The more precise result that $\ZZ [\frac{1+\sqrt{5}}{2}]\subset {\rm End}_{\overline{k}}(J(C_{a,b}))$
follows as well.

{\renewcommand{\arraystretch}{1.3}
$$
\begin{array}{rcl}

\end{array}
$$
}

\end{document}